\documentclass[12pt]{article}
\usepackage{amsmath,amssymb,color,amscd}
\usepackage{xspace}

\def\1{\underline{1}}
\def\R{{\mathbb R}}

\def\Z{{\mathbb Z}}

\def\C{{\mathbb C}}

\def\O{{\mathcal O}}

\newcommand{\ind}{{\rm ind}\,}
\newcommand{\indGSV}{{\rm ind}_{\rm GSV}}
\newcommand{\indrad}{{\rm ind}_{\rm rad}}
\newcommand{\indhom}{{\rm ind}_{\rm hom}}
\newcommand{\rindrad}{{\rm rind}_{\rm rad}}
\newcommand{\rindGSV}{{\rm rind}_{\rm GSV}}
\def\eps{\varepsilon}

\newtheorem{theorem}{Theorem}

\newtheorem{proposition}{Proposition}

\newenvironment{definition}
{\smallskip\noindent{\bf Definition\/}:}{\smallskip\par}
\newenvironment{corollary}
{\smallskip\noindent{\bf Corollary\/}.}{\smallskip\par}

\newenvironment{remark}
{\smallskip\noindent{\bf Remark\/}.}{\smallskip\par}

\newenvironment{proof}
{\noindent{\bf Proof\/}.}{{ $\square$}\smallskip\par}

\title{On equivariant indices of 1-forms on varieties.
}

\author{S.M.~Gusein-Zade \thanks{The work of the first author
(Sections~1, 2, 5, 6) 
was supported by the grant 16-11-10018 of the Russian Science Foundation.
Address: Moscow State University, Faculty
of Mathematics and Mechanics, GSP-1, Moscow, 119991, Russia. E-mail:
sabir\symbol{'100}mccme.ru} \and F.I.~Mamedova \thanks{
Address: Leibniz Universit\"{a}t Hannover, Institut f\"{u}r Algebraische Geometrie,
Postfach 6009, D-30060 Hannover, Germany.
E-mail: mamedova\symbol{'100}math.uni-hannover.de}}
\begin{document}

\maketitle

\begin{abstract}
For a $G$-invariant holomorphic 1-form with an isolated singular point on a germ of a complex-analytic
$G$-variety with an isolated singular point ($G$ is a finite group) one has notions of the equivariant
homological index and of the (reduced) equivariant radial index as elements of the ring of complex
representations of the group. We show that on a germ of a smooth complex-analytic $G$-variety these
indices coincide. This permits to consider the difference between them as a version of the equivariant
Milnor number of a germ a $G$-variety with an isolated singular point.
\end{abstract}

%%%%%%%%%%%%%%%%%%%%%%%%%%%%%%%%%%
{\bf 1. Introduction.}
%%%%%%%%%%%%%%%%%%%%%%%%%%%%%%%%%%
An isolated singular point of a vector field or a 1-form on a smooth manifold has a well-known
integer invariant~-- the index. It can be defined for vector fields or 1-forms on a complex-analytic
manifold as well.
%% (Choice of the sign: \cite{}.)
The notions of the index of an isolated singular
point of a vector field or of a 1-form have generalizations to singular (real or complex) analytic
varieties. One of these generalizations is the radial index defined for an isolated singular point of
a vector field or of a 1-form on an arbitrary (real or complex, singular) analytic variety: \cite{EG-Fields,
EG-Dedicata, EG-Survey}. For a germ $(V,0)$ of a complex analytic variety with an isolated
singular point at the origin and for a complex analytic vector field on it, X.~Gomez Mont defined
the so-called homological index: \cite{G-M}. This notion was generalized to 1-forms in~\cite{EGS}.
The coincidence of the homological index of a holomorphic 1-form and the radial one on a non-singular complex
analytic manifold permits to interpret its difference on a germ of a variety with an isolated singular point
(this difference does not depend on a (complex-analytic) 1-form) as a version of the Milnor number
of the singular point of the variety: \cite{EGS}.

The notion of the radial index has an equivariant version for a singular point of a $G$-invariant
vector field or 1-form on a germ of a variety with an action of a finite group $G$: \cite{EG-EuJM}.
This index takes values in the Burnside ring $A(G)$ of the group. One has a natural homomorphism
from the Burnside ring $A(G)$ to the ring $R(G)$ of (complex) representations of the group $G$.
This gives a version of the equivariant radial index (the reduced equivariant radial index) with
values in the ring $R(G)$.

There are rather natural generalizations of the notions of the homological indices of a vector
field or of a 1-form to the equivariant setting, i.e., for $G$-invariant vector fields or 1-forms
on a germ of a variety with an action of a finite group $G$: see below. These generalizations
have values in the ring $R(G)$ of representations. It is easy to show that, for a holomorphic
vector field on a germ of a smooth complex analytic
manifold, the equivariant homological index and the equivariant ``usual'' (radial) index with values in the
ring $R(G)$ of representations coincide. This follows from the fact that one has a $G$-invariant deformation
of a $G$-invariant vector field with only non-degenerate singular points,
whence for non-degenerate singular points
these two indices obviously coincide. On the other hand similar arguments do not work for 1-forms.
$G$-invariant deformations of a $G$-invariant holomorphic 1-form on $(\C^n,0)$ have, as a rule, complicated singular points.
In order to prove that the equivariant homological index and the equivariant radial index of a holomorphic
1-form coincide, it is possible to try to describe all singular points which can appear in generic $G$-invariant
deformations and to compare these indices for them. However this seems to be a rather involved task in
general. This can be done for particular groups (say, for the cyclic groups $\Z_2$ and $\Z_3$), however
it is not clear to which extent this program can be performed in the general setting.

Here we prove that the equivariant homological index and the reduced equivariant radial index of a singular
point of a holomorphic 1-form on a smooth complex-analytic manifold coincide. The proof is based on
an induction by the dimension of the manifold and by the order of a (cyclic) group. This statement
permits to consider the difference between these indices as a version of the equivariant Milnor number
of a germ of a $G$-variety with an isolated singular point.

The authors are grateful to W.~Ebeling for a careful reading of the manuscript and a number of useful comments.

%%%%%%%%%%%%%%%%%%%%%%%%%%%%%%%%%%%%%%%%%%%%%%%%%
{\bf 2. Equivariant radial and homological indices.}
%%%%%%%%%%%%%%%%%%%%%%%%%%%%%%%%%%%%%%%%%%%%%%%%%
First we recall the notion of the equivariant radial index of a ($G$-invariant) 1-form
on a (real or complex) analytic variety: \cite{EG-EuJM}. Let the space $(\R^N, 0)$ be endowed with
a smooth action of a finite group $G$. Without loss of generality we may assume that the action
is linear. Let $(X,0)\subset(\R^N, 0)$ be a germ of a $G$-invariant real analytic variety at the origin and  
let $\omega$ be a (continuous) $G$-invariant 1-form on $(\R^N, 0)$.
Let $X = \bigcup_{i=0}^q X_i$ be a $G$-invariant Whitney stratification of the germ $(X, 0)$
such that all points $x$ of each stratum have one and the same isotropy group $G_x=\{g\in G:gx=x\}$.
A singular point of the 1-form $\omega$ on $(X,0)$ is a singular point of its
restriction to a stratum of the Whitney stratification of $(X,0)$. (If the stratum is zero-dimensional,
its point is assumed to be singular.) Let us assume that
the 1-form $\omega$ has an isolated singular point at the origin on $(X,0)$.

\begin{definition}
 A 1-form $\omega$ is called {\em radial} on $(X, 0)$ if, for an arbitrary
nontrivial analytic arc $\varphi: (\R, 0)\to (X, 0)$ on $(X, 0)$, the value of the 1-form
$\omega$ on the tangent vector $\dot{\varphi}(t)$ is positive for positive $t$ (small enough).
\end{definition}

Let $\eps>0$ be small enough so that in the closed ball $B_\eps$ of radius $\eps$ centred at the origin
in $\R^N$ the 1-form $\omega$ has no singular points on $X\setminus\{0\}$. One can show that there exists
a $G$-invariant 1-form $\widetilde\omega$ on a neighbourhood of $B_\eps$ possessing the following
properties.
\begin{itemize}
\item[1)] The 1-form $\widetilde\omega$ coincides with $\omega$ on a neighbourhood of
the sphere $S_\eps = \partial B_\eps$. 
\item[2)] The 1-form $\widetilde\omega$ is radial on $(X, 0)$ at the origin. 
\item[3)] In a neighbourhood of each singular point $x_0\in (X\cap B_\eps)\setminus\{0\}$, $x_0 \in
X_i$, $\dim X_i=k$, the 1-form $\widetilde{\omega}$ looks as follows. There exists a
(local) analytic diffeomorphism $h: (\R^N, \R^k, 0) \to (\R^N, X_i, x_0)$ such that
$h^\ast\widetilde{\omega} = \pi_1^\ast \widetilde{\omega}_1 + \pi_2^\ast \widetilde{\omega}_2$,
where $\pi_1$ and $\pi_2$ are the natural projections $\pi_1: \R^N \to \R^k$
and $\pi_2: \R^N \to \R^{N-k}$ respectively, 
$\widetilde{\omega}_1$ is the germ of a 1-form on $(\R^k,0)$ with an isolated singular point
at the origin, and $\widetilde{\omega}_2$ is a radial 1-form on $(\R^{N-k},0)$.
\end{itemize}
The usual index $\ind(\widetilde{\omega}_{\vert X_i}; X_i,p)$ of the restriction
of the 1-form $\widetilde{\omega}$ to the corresponding stratum (a smooth manifold)
will be called the {\em multiplicity} of the 1-form $\widetilde{\omega}$ at the point $x_0$.
(If the origin is a stratum of the stratification itself (the zero-dimensional one),
the multiplicity of $\widetilde{\omega}$ at the origin is assumed to be equal to $1$.)

\begin{definition} \cite{EG-EuJM}
 The {\em equivariant radial index} ${\rm ind}^G_{\rm rad}(\omega;X,0)$ of the 1-form $\omega$ on
 the variety $X$ at the origin is the element of the Burnside ring $A(G)$ of the group $G$ 
 represented by the set of singular points of the 1-form $\widetilde{\omega}$ regarded with
 the multiplicities.
\end{definition}

\begin{remark}
 It is possible to assume that the restrictions of the 1-form $\widetilde{\omega}$ to the strata
 have only non-degenerate singular points. In this case all the multiplicities are equal to $\pm 1$.
\end{remark}

Let the space $(\C^N, 0)$ be endowed with an (analytic) action of a finite group $G$.
(Without loss of generality we may assume that the action is linear.)
Let $(X,0)\subset (\C^N, 0)$ be a germ of a $G$-invariant complex analytic variety of pure
dimension $n$ and let $\omega$ be a (continuous, complex-valued) $G$-invariant 1-form on $(\C^N, 0)$.

\begin{definition}
 The {\em equivariant radial index} ${\rm ind}^G_{\rm rad}(\omega;X,0)$ of the complex 1-form $\omega$ on
 the variety $X$ at the origin is defined by the equation
 $$
 {\rm ind}^G_{\rm rad}(\omega;X,0)=(-1)^n{\rm ind}^G_{\rm rad}({\rm Re\,}\omega;X,0)\in A(G)\,,
 $$
 where ${\rm Re\,}\omega$ is the real part of the 1-form $\omega$ (see the explanation of the sign, e.g.,
 in \cite{EGS}).
\end{definition}

One has a natural homomorphism $r_G:A(G)\to R(G)$ (``reduction'') sending a finite $G$-set $X$ to
the space of (complex valued) functions on it with the induced representation of the group $G$.

\begin{definition}
 The {\em reduced equivariant radial index} ${\rm rind}^G_{\rm rad}(\omega;X,0)$ of 
 a (real or complex) 1-form $\omega$ on a (real or complex) analytic variety $(X,0)$ 
 is
 $$
 {\rm rind}^G_{\rm rad}(\omega;X,0)=r_G\left({\rm ind}^G_{\rm rad}(\omega;X,0)\right)\in R(G)\,.
 $$
\end{definition}

As above, let the space $(\C^N, 0)$ be endowed with a linear action of a finite group $G$ and
let $(X,0)\subset (\C^N, 0)$ be a germ of a $G$-invariant complex analytic variety of pure
dimension $n$. Let us assume that $X$ has an isolated singular point at the origin.
Let $\omega$ be a $G$-invariant holomorphic 1-form on $(X,0)$ (that is the restriction to $(X,0)$
of a ($G$-invariant) holomorphic 1-form on $(\C^N, 0)$) without singular points (zeroes) outside of the origin.
Let us consider the complex $(\Omega^{\bullet}_{X,0}, \wedge\omega)$: 
\begin{equation}\label{hom-complex}
0 \to \Omega^0_{X,0}
\to \Omega^1_{X,0} \to ... \to \Omega^n_{X,0} \to 0\,, 
\end{equation}
where $\Omega^i_{X,0}$ are the modules of germs of differential $i$-forms on $(X,0)$
($\Omega^0_{X,0}={\mathcal O}_{X,0}$)
and the arrows are %%given by 
the exterior products by the 1-form $\omega$: $\wedge\omega$.
This complex has finite-dimensional cohomology groups $H^i(\Omega^\bullet_{X,0},\wedge\omega)$.
(This follows from the fact that the corresponding complex of sheaves consists of coherent
sheaves and its cohomologies are concentrated at the origin.)
All the spaces $\Omega^i_{X,0}$ and thus the cohomology groups $H^i(\Omega^\bullet_{X,0},\wedge\omega)$
carry natural representations of the group $G$. The definition of the ``usual'' (non-equivariant)
homological index of a 1-form from \cite{EGS} inspires the following definition.

\begin{definition} 
The {\em equivariant homological index}
$\,{\rm ind}_{\rm hom}^G(\omega; X,0)$ of the 1-form %% ={\rm ind}_{\rm hom}^G\, \omega
$\omega$ on $(X, 0)$ is defined by the equation
%% $(-1)^n$ times the Euler characteristic of the above complex:
\begin{equation}\label{eq1}
{\rm ind}_{\rm hom}(\omega; X,0) %% ={\rm ind}_{\rm hom}\, \omega
= \sum_{i=0}^n (-1)^{n-i} [H^i(\Omega^\bullet_{X,0},\wedge\omega)]\in R(G)\,,
\end{equation}
where $[H^i(\Omega^\bullet_{X,0},\wedge\omega)]$ is the class of the (finite-dimensional)
$G$-module $H^i(\Omega^\bullet_{X,0},\wedge\omega)$ in the ring $R(G)$ of complex
representations of the group $G$.
\end{definition}

 The equivariant homological index satisfies the following law of conservation of number.
 Let $\omega'$ be a small $G$-invariant holomorphic deformation of the 1-form $\omega$.
 For a singular point $x$ of the 1-form $\omega'$ in a punctured neighbourhood
 of the origin $0$ in $X$, let $G_x=\{g\in G:gx=x\}$ be the isotropy subgroup of
 the point $p$ and let ${\rm ind}_{\rm hom}^{G_x}(\omega';X,x)\in R(G_x)$ be the
 equivariant homological index of the 1-form $\omega'$ at the point $x$.
 For a subgroup $H\subset G$, one has the natural (linear) map $I_H^G:R(H)\to R(G)$:
 the induction map (not a ring homomorphism). 
 
\begin{proposition}
One has the equation
$$
{\rm ind}_{\rm hom}^G(\omega; X,0) = {\rm ind}_{\rm hom}^G(\omega'; X,0)
 + \sum_{[x]\in (X\setminus\{0\})/G} I_{G_x}^G\left({\rm ind}_{\rm hom}^{G_x}(\omega';X,x)\right)\,,
$$ 
where the sum on the right hand side is over all orbits $[x]$ of singular points of the
1-form $\omega'$ in a small punctured neighbourhood of the origin $0$ in $X$, $x$ is a
representative of the orbit $[x]$.
\end{proposition}

The proof can be obtained from the proof (of a more general statement) in \cite{GG}
by considering all the sheaves and modules there with the corresponding actions 
(representations) of the group $G$. If $X$ is non-singular (i.e.\ $(X,0)\cong(\C^n,0)$),
the only non-trivial cohomology group of the complex $(\Omega^{\bullet}_{X,0}, \wedge\omega)$
is in the dimension $n$. (In fact the same holds if $(X,0)$ is an isolated complete intersection
singularity: see Section~7.) If the 1-form $\omega$ on $(\C^n,0)$ is equal to
$\sum\limits_{i=1}^n f_ndz_n$, one has
$$
{\rm ind}_{\rm hom}^G(\omega; \C^n,0)=
[(\O_{\C^n,0}/\langle f_1,\ldots,f_n\rangle)dz_1\wedge\ldots\wedge dz_n]\in R(G).
$$
In this case the statement can be reduced to an equivariant version of the law of conservation
of number for the multiplicity $\dim_{\C}(\O_{\C^n,0}/\langle f_1,\ldots,f_n\rangle)$ of the map
$F=(f_1, \ldots, f_n):(\C^n,0)\to(\C^n,0)$.
A proof of the equivariant version can be obtained by an appropriate modification of a proof
of the traditional (non-equivariant) version, say, of the one given in \cite[Section~5]{AGV}.

%%%%%%%%%%%%%%%%%%%%%%%%%%%%%%%%%%%%%%%%%%%%%%%%%%%%%%%%%%%%%%%%%%%%%%%%%%%%%%%%%%%%%%
{\bf 3. Equivariant radial and homological indices in the one-di\-men\-sio\-nal case.} %% Done.
%%%%%%%%%%%%%%%%%%%%%%%%%%%%%%%%%%%%%%%%%%%%%%%%%%%%%%%%%%%%%%%%%%%%%%%%%%%%%%%%%%%%%%
A finite group $G$ acting faithfully on the line $(\C,0)$ is a cyclic one, say, $\Z_m$.
Let $\sigma$ be a generator of $\Z_m$. Without loss of generality we can assume that $\sigma$
acts on $\C$ by multiplication by $\sigma:=\exp(2\pi i/m)$. (The coincidence of notations for
a generator of $\Z_m$ and for $\exp(2\pi i/m)$ here and below does not lead to a confusion.
Moreover, we shall use the same notation for the described representation of the group $\Z_m$ on $\C$.)
A (non-trivial) $\Z_m$-invariant 1-form on $(\C,0)$ is right-equivalent to $z^{sm-1}dz$ (i.e., can
be reduced to this one by a change of the variable on $\C$).

\begin{proposition}
 The reduced radial and the homological equivariant indices of the 1-form $\omega_s=z^{sm-1}dz$
 (as elements of the ring $R(\Z_m)$) are equal to $s(1+\sigma+\sigma^2+\ldots+\sigma^{m-1})-1$.
\end{proposition}

\begin{proof}
 The usual (non-equivariant) index of this (complex) 1-form is equal to $sm-1$. Therefore the index
 of its real part is equal to $1-sm$. A $G$-equivariant 1-form $\widetilde{\omega}_s$ from the
 definition of the radial index of the 1-form ${\rm Re\,}\omega_s$ is radial at the origin of $\C\cong\R^2$
 and has free orbits of singular points outside of it. Therefore
 $$
 \rindrad^G({\rm Re\,}\omega_s; \R^2,0)= 1 - sI_{(e)}^{\Z_m}(1)=1-s(1+\sigma+\sigma^2+\ldots+\sigma^{m-1})\,.
 $$
 Thus $\rindrad^G(\omega_s; \C,0)=s(1+\sigma+\sigma^2+\ldots+\sigma^{m-1})-1$.
 
 A basis of $\Omega^1_{\C,0}/\omega_s\wedge \Omega^0_{\C,0}$ consists of the (monomial) 1-forms
 $dz$, $zdz$, \dots, $z^{sm-2}dz$. On the element $z^jdz$ the generator $\sigma$ acts by the representation
 $\sigma^{j+1}$. This gives $\indhom^G(\omega_s; \C,0)=s(1+\sigma+\sigma^2+\ldots+\sigma^{m-1})-1$.
\end{proof}

\begin{proposition}
 The reduced radial and the homological equivariant index of the 1-form $\omega_s=z^{sm-1}dz$, i.\ e.,
${\rm ind}=s(1+\sigma+\sigma^2+\ldots+\sigma^{m-1})-1$, is not a divisor of zero in $R(\Z_m)$.
\end{proposition}

\begin{proof}
 The table of the multiplication of the basis elements $\sigma^i$, $i=0,1,\ldots, m-1$ by the element
 ${\rm ind}$ is given by the $(m\times m)$-matrix $sI-E$, where $I$ is the matrix all whose entries are equal
 to $1$, $E$ is the unit matrix. This matrix is non-degenerate (since its eigenvalues are $s(m-1)$ and $(-1)$,
 the latter one with the multiplicity $(m-1)$).
\end{proof}

%%%%%%%%%%%%%%%%%%%%%%%%%%%%%%%%%%%%%%%%%%%%%%%%%%%%%%%%%%%%%
{\bf 4. Sebastiani--Thom formula for the equivariant indices.} %% Done
%%%%%%%%%%%%%%%%%%%%%%%%%%%%%%%%%%%%%%%%%%%%%%%%%%%%%%%%%%%%%
Let $\C^n$ and $\C^m$ be spaces with actions (representations) of the group $G$ and let
$\omega$ and $\eta$ be $G$-invariant 1-forms on $(\C^n,0)$ and on $(\C^m,0)$ respectively
with isolated singular points at the origin.
One has the Sebastiani--Thom (direct) sum $\omega\oplus\eta$ of the 1-forms $\omega$ and $\eta$
(a 1-form on $(\C^n\oplus\C^m,0)\cong (\C^{n+m},0)$)
defined by the equation $(\omega\oplus\eta)_{(x,y)}(u,v)=\omega_x(u)+\eta_y(v)$ ($x\in\C^n$, $y\in\C^m$,
$u\in T_x\C^n\cong\C^n$, $v\in T_y\C^m\cong\C^m$).

\begin{theorem}\label{S-T} (a version of the Sebastiani--Thom theorem)
 One has the equations
\begin{eqnarray*}
 \indrad^{G}(\omega\oplus\eta; \C^{n+m},0)&=&\indrad^{G}(\omega; \C^n,0)\cdot\indrad^{G}(\eta; \C^m,0)\in A(G)\,,\\
 \indhom^{G}(\omega\oplus\eta; \C^{n+m},0)&=&\indhom^{G}(\omega; \C^n,0)\cdot\indhom^{G}(\eta; \C^m,0)\in R(G)\,.
\end{eqnarray*}
\end{theorem}

\begin{proof}
 For the radial index this follows from the following construction. Let $\widetilde{\omega}$ and
 $\widetilde{\eta}$ be 1-forms described in the definition of the equivariant radial index (corresponding
 to the 1-forms $\omega':={\rm Re\,}\omega$ and $\eta':={\rm Re\,}\eta$ respectively).
 Without loss of generality we may assume that
 $\widetilde{\omega}$ and $\widetilde{\eta}$ are defined on the balls $B_{\eps}^{2n}$ and $B_{\eps}^{2m}$
 (centred at the origin) in $\C^n$ and $\C^m$ respectively of the same radius $\eps$ and that they coincide
 with $\omega'$ and $\eta'$ respectively outside of the balls of radius $\eps/4$.
 Let $\psi(r)$ be a (continuous) function on $[0,\eps]$ such that $0\le\psi(r)\le 1$, $\psi(r)\equiv 1$
 for $r\le\eps/2$, $\psi(r)\equiv 0$ for $r\ge 3\eps/4$.
 Let us define a 1-form $\widetilde{\omega\oplus\eta}$ on $B_{\eps}^{2(n+m)}\subset\C^{n+m}$
 by the equation
 %% $$
 %% \widetilde{\omega\oplus\eta}_{(x,y)}=
 %% \begin{cases}
 %%  \widetilde{\omega}_x\oplus\widetilde{\eta}_y &{\rm for\ } \Vert x\Vert\le \eps/3, \Vert y\Vert\le \eps/3,\\
 %%  \left(\lambda_y{\rm Re\,}\omega_x+
 %%  (1-\lambda_y)\widetilde{\omega}_x\right)\oplus\widetilde{\eta}_y
 %%  &{\rm for\ } \Vert x\Vert\le \eps/3, \eps/3\le\Vert y\Vert\le 2\eps/3,\\
 %%  \widetilde{\omega}_x\oplus\left(\lambda_x{\rm Re\,}\eta_y+
 %%  (1-\lambda_x)\widetilde{\eta}_y\right)
 %%  &{\rm for\ } \Vert y\Vert\le \eps/3, \eps/3\le\Vert x\Vert\le 2\eps/3,\\
 %%  {\rm Re\,}\omega_x\oplus{\rm Re\,}\eta_y
 %%  &{\rm for\ other\ values\ of\ } 
 %%  \Vert x\Vert {\rm\ and\ }\Vert y\Vert,
 %% \end{cases}
 %% $$
 %% where $x\in\C^n$, $y\in\C^m$ $\lambda_{\bullet}:=\frac{3\Vert \bullet\Vert}{\eps}-1$.
 $$
 \widetilde{\omega\oplus\eta}_{(x,y)}=
 (1-\psi(r))\omega'_x\oplus\eta'_y+\psi(r)\widetilde{\omega}_x\oplus\widetilde{\eta}_y\,,
 %% \begin{cases}
 %% \widetilde{\omega}_x\oplus\widetilde{\eta}_y &{\rm for\ } r\le \eps/2,\\
 %% \omega'_x\oplus\eta'_y &{\rm for\ } r\ge 3\eps/4,\\
 %% (4r/\eps-2)\omega'_x\oplus\eta'_y+(3-4r/\eps)\widetilde{\omega}_x\oplus\widetilde{\eta}_y
 %% &{\rm for\ } \eps/2\le r\le 3\eps/4,
 %% \end{cases}
 $$
 where $x\in\C^n$, $y\in\C^m$, $r:=\sqrt{\Vert x\Vert^2+\Vert y\Vert^2}$.
 One can see that the 1-form $\widetilde{\omega\oplus\eta}$ considered on the ball
 $B_{\eps}^{2(n+m)}\subset\C^{n+m}$ is appropriate for the definition of the equivariant
 radial index of the 1-form ${\rm Re\,}(\omega\oplus\eta)=\omega'\oplus\eta'$ (i.\ e., satisfies
 the conditions 1--3 above). Moreover, the set of its singular points (considered as a $G$-set)
 is the direct product of the sets of singular points of the 1-forms $\widetilde{\omega}$
 and $\widetilde{\eta}$. (This follows from the fact that, for a point $(x,y)$ outside of
 $B_{\eps/2}^{2n}\times B_{\eps/2}^{2m}$ either $\widetilde{\omega}_x=\omega'_x$ or
 $\widetilde{\eta}_y=\eta'_y$ and therefore the 1-form $\widetilde{\omega\oplus\eta}$ does not vanish.)
 
 For the homological index this follows from
 the fact that for a 1-form on a non-singular manifold the only non-trivial cohomology group of the complex
 (\ref{hom-complex}) is in the highest dimension and one has
 $$
 \Omega^{n+m}_{\C^{n+m},0}/(\omega\oplus\eta)\wedge\Omega^{n+m-1}_{\C^{n+m},0}=
 (\Omega^{n}_{\C^{n},0}/\omega\wedge\Omega^{n-1}_{\C^{n},0})\otimes
 (\Omega^{m}_{\C^{m},0}/\eta\wedge\Omega^{m-1}_{\C^{m},0})
 $$
 (as spaces with $G$-representations).
\end{proof}

\begin{remark}
For the radial index the same equation holds for two 1-forms on (singular) varieties
and for the corresponding 1-form (the direct sum) on the product of the varieties.
For the homological index defined here, the corresponding equation does not make sense.
Here the homological index is defined for a 1-form on a variety with an isolated singular point,
whence the product of two varieties with isolated singular points has non-isolated singular points.
\end{remark}

\begin{corollary}
 One has
$$
\rindrad^{G}(\omega\oplus\eta; \C^{n+m},0)=\rindrad^{G}(\omega; \C^n,0)\cdot\rindrad^{G}(\eta; \C^m,0)\in R(G)\,.
$$
\end{corollary}

%%%%%%%%%%%%%%%%%%%%%%%%%%%%%%%%%%%%%%%%%%%%%%%%%%%%%%%%%%%%%%%%%%%%%%%%%%%%%%%%%%%%%%
{\bf 5. Destabilization of singular points.}
%%%%%%%%%%%%%%%%%%%%%%%%%%%%%%%%%%%%%%%%%%%%%%%%%%%%%%%%%%%%%%%%%%%%%%%%%%%%%%%%%%%%%%
Let $\C^m=\C^{m-k}\oplus\C^{k}$ be a $G$-invariant decomposition of the space $\C^m$ with a
representation of the group $G$ such that the $k$th exterior power of the action of $G$ on $\C^{k}$
(i.e., the action of $G$ on the space of $k$-forms on $\C^{k}$) is trivial. Let $\omega$ be a
$G$-invariant holomorphic 1-form on $(C^m,0)$ such that its restriction to $(\C^{k},0)$ is non-degenerate.
%% If $G$ is a cyclic group (the case used below), this can hold only if the representation
%% of the group $G$ on $\C^{2}$ is the direct sum ???

\begin{proposition}\label{prop1}
 There exists a $G$-invariant complex analytic 1-form $\eta$ on $(\C^{m-k},0)$ such that
 $\indhom^{G}(\omega; \C^m,0)=\indhom^{G}(\eta; \C^{m-k},0)$,
 $\indrad^{G}(\omega; \C^m,0)=\indrad^{G}(\eta; \C^{m-k},0)$, and therefore
 $\rindrad^{G}(\omega; \C^m,0)=\rindrad^{G}(\eta; \C^{m-k},0)$.
\end{proposition}

\begin{proof}
 For small $x\in\C^{m-k}$, the restriction of the 1-form $\omega$ to the affine subspace
 $\{x\}\times\C^{k}$ has one non-degenerate zero $\{x\}\times f(x)$ in a neighbourhood
 of $\{x\}\times\{0\}$, where $f$ is a $G$-equivariant analytic map from $(\C^{m-k},0)$
 to $(\C^{k},0)$. Let $H:\C^{m-k}\oplus\C^{k}\to\C^{m-k}\oplus\C^{k}$ be defined by
 $H(x,y)=(x,y+f(x))$. The map $H$ is a local $G$-equivariant holomorphic automorphism of
 $\C^{m-k}\oplus\C^{k}$.
 The 1-form $H^*\omega$ has the same equivariant radial and homological indices as $\omega$.
 Moreover, for any $x\in(\C^{m-k},0)$, the restriction of $H^*\omega$ to $\{x\}\times\C^{k}$
 has a non-degenerate singular point at the origin $\{x\}\times\{0\}$. If $\varphi_i(\overline{z})$,
 $i=1,\ldots, m$, are the components of the 1-form $H^*\omega$
 ($H^*\omega=\sum\limits_{i=1}^m\varphi_i(\overline{z})dz_1$), then the ideal in $\O_{\C_m,0}$ generated
 by $\varphi_{m-k+1}(\overline{z})$, \dots, $\varphi_{m}(\overline{z})$ coincides with the ideal
 $\langle z_{m-k+1}, \ldots, z_m\rangle$. Therefore
 $$
 \O_{\C^m,0}/\langle \varphi_1, \ldots, \varphi_m\rangle=
 \O_{\C^{m-k},0}/\langle {\varphi_1}_{\vert(\C^{m-k},0)}, \ldots, {\varphi_{m-k}}_{\vert(\C^{m-k},0)}\rangle\,.
 $$
 This implies that $\indhom^{G}(\omega; \C^m,0)=\indhom^{G}((H^*\omega)_{\vert(\C^{m-k},0)}; \C^{m-k},0)$.
 
 Let $\pi_1$ and $\pi_2$
 be the natural projections of $T_p\C^m\cong\C^m$ to $T_p\C^{m-k}\cong\C^{m-k}$ and to $T_p\C^k\cong\C^k$
 respectively ($p\in(\C_m,0)$). Let $\omega_i=\pi_i^*H^*\omega$, $i=1,2$. One has
 $H^*\omega=\omega_1+\omega_2$. (Pay attention that $\pi_1$
 and $\pi_2$ are not maps from $(\C^m,0)$ to $(\C^{m-k},0)$ and to $(\C^k,0)$.)
 Let $\eps>0$ be small enough and let $\psi(r)$ be a function as described in the proof of Theorem~\ref{S-T}.
 Let $\widehat{\omega}$ be the 1-form defined by
 $\widehat{\omega}_{\vert(\overline{z}',\overline{z}'')}$
 $=
 {\omega_1}_{\vert(\overline{z}',0)}+{\omega_2}_{\vert(0,\overline{z}'')}$,
 where $\overline{z}'\in(\C^{m-k},0)$, $\overline{z}''\in(\C^{k},0)$.
 One can see that the 1-form $\psi(r)\widehat{\omega}+(1-\psi(r))H^*\omega$ has no
 zeroes in the ball of radius $\eps$ outside of the origin, coincides with $H^*\omega$
 in a neighbourhood of the boundary of the ball and coincides with
 $(H^*\omega)_{(\C^{m-k},0)}\oplus (H^*\omega)_{(\C^{k},0)}$ in the ball of radius $\eps/2$.
 According to Theorem~\ref{S-T} this implies that 
 $\indrad^G(\omega;\C^{m},0)=\indrad^G((H^*\omega)_{(\C^{m-k},0)};\C^{m-k},0)\cdot
 \indrad^G((H^*\omega)_{(\C^{k},0)};\C^{k},0)=\indrad^G((H^*\omega)_{(\C^{m-k},0)};\C^{m-k},0)$.
\end{proof}

%%%%%%%%%%%%%%%%%%%%%%%%%%%%%%%%%%%%%%%%%%%%%%%%%%%%%%%%%%%%%%%%%%%%%%%%%%%%%%%%%%%%%%
{\bf 6. Coincidence of equivariant radial and homological indices on smooth manifolds.}
%%%%%%%%%%%%%%%%%%%%%%%%%%%%%%%%%%%%%%%%%%%%%%%%%%%%%%%%%%%%%%%%%%%%%%%%%%%%%%%%%%%%%%
We are ready to prove the main statement %% result
of the paper.

\begin{theorem}
 For a $G$-invariant holomorphic 1-form $\omega$ on $(\C^n,0)$ one has
 $$
 \rindrad^{G}(\omega; \C^n,0)=\indhom^{G}(\omega; \C^n,0)\,.
 $$
\end{theorem}

\begin{proof}
 For a subgroup $H$ of the group $G$, the indices $\rindrad^{H}(\omega; \C^n,0)$ and 
 $\indhom^{H}(\omega; \C^n,0)$ are the images of the indices $\rindrad^{G}(\omega; \C^n,0)$
 and $\indhom^{G}(\omega; \C^n,0)$ under the reduction homomorphism $R^G_H$. A representation
 of a finite group is determined by its character: the trace of the corresponding operator
 as a function on the group. Each element of a finite group is contained in a cyclic subgroup.
 Therefore it is sufficient to prove the statement for $G$ being a cyclic group $\Z_d$.

 The proof will use the induction both on the dimension $n$ of the space and on the number $d$ of
 elements of the group $G=\Z_d$. For $n=1$ (the 1-dimensional case) the statement is proved in Section~3.
 For the trivial group $G$ (i.\ e., in the non-equivariant setting) the statement is well known
 (see, e.\ g., \cite{EGS}). Assume first that the representation of the group $G$ on $\C^n$ has
 a non-trivial summand $\C^k$ with the trivial representation of $G$, $\C^n=\C^{n-k}\oplus\C^k$ is a
 decomposition of the representation on $\C^n$. There exists a $G$-invariant holomorphic deformation
 $\widetilde{\omega}$ of the 1-form $\omega$ such that at each singular point (zero) $p$ of the
 1-form $\widetilde{\omega}$ with $p\in\{0\}\times\C^k$ (i.e., $p=(0,y_0)$) the restriction of
 $\widetilde{\omega}$ to the (affine) subspace $\{0\}\times\C^k$ has a non-degenerate zero at
 $p$. Proposition~\ref{prop1} implies that there exists a $G$-invariant 1-form $\omega'$ on 
 $(\C^{n-k},0)$ such that $\rindrad^{G}(\widetilde{\omega}; \C^n,p)=\rindrad^{G}(\omega'; \C^{n-k},0)$,
 $\indhom^{G}(\widetilde{\omega}; \C^n,p)=\indhom^{G}(\omega'; \C^{n-k},0)$. According to the assumption
 of the induction one has $\rindrad^{G}(\omega'; \C^{n-k},0)=\indhom^{G}(\omega'; \C^{n-k},0)$
 and therefore $\rindrad^{G}(\widetilde{\omega}; \C^n,p)=\indhom^{G}(\widetilde{\omega}; \C^n,p)$.
 For a singular point
 $p$ of the 1-form $\widetilde{\omega}$ outside of $\{0\}\times\C^k$, one has $G_p\varsubsetneq G$.
 The assumption of the induction gives
 $\rindrad^{G_p}(\widetilde{\omega}; \C^n,p)=\indhom^{G_p}(\widetilde{\omega}; \C^n,p)$ and therefore
 $I_{G_p}^G\rindrad^{G_p}(\widetilde{\omega}; \C^n,p)=I_{G_p}^G\indhom^{G_p}(\widetilde{\omega}; \C^n,p)$.
 The laws of conservation of number for the equivariant radial and for the equivariant homological
 indices imply that
 $\rindrad^{G}(\omega; \C^n,0)=\indhom^{G}(\omega; \C^n,0)$. Therefore we can assume that the
 representation of $G$ on the space $\C^n$ has no trivial summands.
 
 Let $\sigma$ be a generator of the group $G=\Z_d$ and let $\sigma$ act on $\C^n$ by
 $\sigma\star(z_1,\ldots,z_n)=(\sigma^{k_1}z_1,\ldots,\sigma^{k_n}z_n)$, where (in the RHS)
 $\sigma=\exp{\frac{2\pi i}{d}}$, $0<k_i<d$ for $i=1,\ldots,n$. Let the space $\C^{n+1}=\C^{n}\oplus\C^{1}$
 be endowed with the representation $\sigma\star(z_1,\ldots,z_n, z_{n+1})=
 (\sigma^{k_1}z_1,\ldots,\sigma^{k_n}z_n,\sigma^{-k_n}z_{n+1})$ of the group $\Z_d$ and let
 $\widehat{\omega}=\omega\oplus z_{n+1}^{d-1}dz_{n+1}$. One has
 $\rindrad^{G}(\widehat{\omega}; \C^{n+1},0)=
 \rindrad^{G}(\omega; \C^n,0)\cdot\rindrad^{G}(z_{n+1}^{d-1}dz_{n+1}; \C^1,0)$,
 $\indhom^{G}(\widehat{\omega}; \C^{n+1},0)=
 \indhom^{G}(\omega; \C^n,0)\cdot\indhom^{G}(z_{n+1}^{d-1}dz_{n+1}; \C^1,0)$ (Theorem~\ref{S-T}).
 Since $\rindrad^{G}(z_{n+1}^{d-1}dz_{n+1};\C^1,0)=\indhom^{G}(z_{n+1}^{d-1}dz_{n+1}; \C^1,0)$
 is not a divisor of zero (Section~3), it is sufficient to show that
 $\rindrad^{G}(\widehat{\omega}; \C^{n+1},0)=\indhom^{G}(\widehat{\omega}; \C^{n+1},0)$.
 Let $\widehat{\omega}':=\widehat{\omega}+\lambda (z_{n+1}dz_{n}+z_{n}dz_{n+1})$ be a deformation of
 the 1-form $\widehat{\omega}$ ($\lambda$ is small enough). The restriction of the 1-form $\widehat{\omega}'$
 to the subspace $\C^2$ corresponding to the last two coordinates has a non-degenerate singular point
 at the origin. By Proposition~\ref{prop1} there exists a holomorphic 1-form $\eta$ on $(\C^{n-1},0)$
 such that $\indhom^{G}(\widehat{\omega}'; \C^{n+1},0)=\indhom^{G}(\eta; \C^{n-1},0)$,
 $\rindrad^{G}(\widehat{\omega}'; \C^{n+1},0)=\rindrad^{G}(\eta; \C^{n-1},0)$.
 According to the assumption
 of the induction one has $\rindrad^{G}(\eta; \C^{n-1},0)=\indhom^{G}(\eta; \C^{n-1},0)$.
 For a singular point
 $p$ of the 1-form $\widehat{\omega}'$ outside of the origin, one has $G_p\varsubsetneq G$.
 The assumption of the induction gives $\rindrad^{G_p}(\widehat{\omega}'; \C^{n+1},p)=
 \indhom^{G_p}(\widehat{\omega}'; \C^{n+1},p)$
 and therefore $I_{G_p}^G\rindrad^{G_p}(\widehat{\omega}'; \C^{n+1},p)=
 I_{G_p}^G\indhom^{G_p}(\widehat{\omega}'; \C^{n+1},p)$.
 The laws of conservation of number for the equivariant radial and for the equivariant homological indices
 imply that $\rindrad^{G}(\widehat{\omega}; \C^{n+1},0)=\indhom^{G}(\widehat{\omega}; \C^{n+1},0)$.
\end{proof}

%%%%%%%%%%%%%%%%%%%%%%%%%%%%%%%%%%%%%%%%%%%%%%%%%%%%%%%%%%%%%%%%%%%%%%%%%%%%%%%
{\bf 7. An equivariant version of the Milnor number for singular varieties.}
%%%%%%%%%%%%%%%%%%%%%%%%%%%%%%%%%%%%%%%%%%%%%%%%%%%%%%%%%%%%%%%%%%%%%%%%%%%%%%%
A notion of the GSV-index of a (continuous) 1-form on an isolated (complex) complete intersection singularity
(ICIS) was introduced in \cite{EG-MMJ}. There was given an algebraic formula for the GSV-index of a
holomorphic 1-form. (The proof there contained a minor mistake corrected in \cite[Theorem 4]{EG-BLMS}.)
In \cite{EGS} it was shown that in this case the GSV-index coincides with the homological one. Actually this
follows directly from the algebraic formula for the GSV-index from \cite{EG-MMJ} and the fact that for
a holomorphic 1-form $\omega$ with an isolated singular point on an $n$-dimensional ICIS $(X,0)$ the only
non-trivial (co)homology group of the complex (\ref{hom-complex}) is the one in dimension $n$:
\cite{Greuel}. Strictly speaking, in \cite{Greuel} it is proved for $\omega=df$, where $f$ is a holomorphic
function on $(X,0)$, however G.-M.~Greuel explained that there is no difference for the general case.

Let $(X,0)=\{f_1=\ldots=f_k=0\}\subset(\C^{n+k}, 0)$ be a $G$-invariant ICIS defined by equations
with $G$-invariant RHSs $f_i$. The notion of the equivariant GSV-index of a $G$-invariant 1-form
$\omega$ on $(X,0)$ was given in \cite{EG-EuJM}. It was defined as an element of the Burnside ring $A(G)$
of the group $G$. One way to define it is the following. Let us take a $G$-invariant representative
of the 1-form $\omega$ defined in a neighbourhood of the origin in $\C^{n+k}$. We shall denote it
by $\omega$ as well. Let
$X_{\overline{\eps}}={\overline{f}}^{-1}(\overline{\eps})\cap B^{2(n+k)}_\delta(0)$ be the Milnor fibre
of the ICIS $(X,0)$ ($\overline{f}=(f_1,\ldots, f_k)$, $\overline{\eps}=(\eps_1,\ldots,\eps_k)$,
$0<\Vert\overline{\eps}\Vert\ll\delta$, $\delta$ is small enough). One may assume that the set
${\rm Sing}\,\omega$ of the singular points of the restriction of the 1-form $\omega$ to
$X_{\overline{\eps}}$ is finite (i.e., this restriction has only isolated singular points (zeroes)). Then
$$
\indGSV^G(\omega;X,0):=
\sum_{[p]\in {\rm Sing}\,\omega/G}I_{G_p}^G(\indrad^{G_p}(\omega;X_{\overline{\eps}},0))\,,
$$
where $p$ is a representative of the $G$-orbit $[p]$.
Let $\rindGSV(\omega;X,0):=r\left(\indGSV(\omega;X,0)\right)\in R(G)$ be the reduction of the equivariant
GSV-index to the ring $R(G)$ of representations of $G$.

The arguments of \cite[Theorem 4]{EG-BLMS} (together with the fact that the only
non-trivial (co)hohomology group of the complex (\ref{hom-complex}) is the one in dimension $n$) imply 
the following statement.

\begin{proposition}
For a holomorphic $G$-invariant 1-form $\omega$ on the $G$-invariant ICIS $(X,0)$,
the equivariant homological index $\indhom^G(\omega,X,0)$ is equal to the reduction
$\rindGSV(\omega;X,0)$ of the equivariant GSV-index. 
\end{proposition}

Let $\chi^G(X_{\overline{\eps}})\in A(G)$ be the equivariant Euler characteristic
of the Milnor fibre $X_{\overline{\eps}}$ and let
$\overline{\chi}^G(X_{\overline{\eps}}):=\chi^G(X_{\overline{\eps}})-1$ be the 
{\em reduced equivariant Euler characteristic} of it.
For a $G$-invariant radial (real) 1-form $\omega_{\rm rad}$ on the ICIS $(X,0)$,
the equivariant GSV-index $\indGSV^G(\omega_{\rm rad};X,0)$ is equal to $\chi^G(X_{\overline{\eps}})$.
This implies the following statement (an equivariant analogue of \cite[Proposition~5.3]{EG-EuJM} for 1-forms).

\begin{proposition}
 For a $G$-invariant real 1-form $\omega$ on the ICIS $(X,0)$ one has
 $$
 \indGSV^G(\omega;X,0)-\indrad^G(\omega;X,0)=\overline{\chi}^G(X_{\overline{\eps}})\,.
 $$
 For a $G$-invariant complex 1-form $\omega$ on the ICIS $(X,0)$ one has
 $$
 \indGSV^G(\omega;X,0)-\indrad^G(\omega;X,0)=(-1)^{n}\overline{\chi}^G(X_{\overline{\eps}})\,.
 $$
\end{proposition}

The reduction $r_G((-1)^{n}\overline{\chi}^G(X_{\overline{\eps}}))\in R(G)$ is the equivariant
Milnor number of the ICIS $(X,0)$ in the sense of \cite{Wall}, i.e., it is equal to the class
in $R(G)$ of the $G$-module $H^n(X_{\overline{\eps}})$.

Let $(X,0)$ be a complex analytic $G$-variety of pure dimension $n$ with an isolated singular point at
the origin. The laws of conservation of number for the equivariant radial and for the equivariant homological
indices together with the fact that they coincide on a smooth manifold imply the following statement.

\begin{proposition}\label{prop2}
 For a $G$-invariant holomorphic 1-form $\omega$ on $(X,0)$ with an isolated singular point at the origin
 the difference $\indhom^G(\omega;X,0) - \rindrad^G(\omega;X,0)\in R(G)$ does not depend on the 1-form
 $\omega$.
\end{proposition}

As it was shown above, for a $G$-invariant ICIS this difference is the equivariant Milnor number
of the ICIS. This permits to regard $\indhom^G(\omega;X,0) - \rindrad^G(\omega;X,0)\in R(G)$
as a version of the equivariant Milnor number of a germ a $G$-variety $(X,0)$ with an isolated
singular point.

\end{document}